\documentclass[11pt,a4paper]{amsart}
\usepackage{amsmath}
\usepackage{amssymb}
\usepackage{mathrsfs}
\usepackage{amscd,mathabx}
\usepackage{comment}
\usepackage{enumitem}
\usepackage{url}
\usepackage[all]{xypic}

\usepackage{tikz}
\usetikzlibrary{patterns,intersections}
\usepackage{graphicx}
\usepackage{tikz-cd} 
\usepackage{float}
\usepackage{bm}
\usepackage{todonotes}
\usepackage[a4paper, top=3cm, bottom=3cm, left=2.5cm, right=2.5cm, marginpar=.75in]{geometry}
\def\R{\mathbb R}

\def\cK{\mathcal{K}}
\def\cX{\mathcal{X}}
\def\cZ{\mathcal{Z}}
\def\cY{\mathcal{Y}}

\def\cC{\mathcal{C}}
\def\cT{\mathcal{T}}

\def\cI{\mathcal{I}}
\def\cG{\mathcal{G}}

\def\cE{\mathcal{E}}

\def\wt#1{\widetilde{#1}}
\def\wh#1{\widehat{#1}}
\def\ol#1{\overline{#1}}

\DeclareMathOperator{\dist}{dist}
\DeclareMathOperator{\vol}{vol}

\DeclareMathOperator{\grf}{graph}

\def\cCamn{\cC^{1,\alpha}_{m,n}}
\def\Coa{C^{1,\alpha}}

\numberwithin{equation}{section}

\theoremstyle{plain}
\newtheorem{theorem}[equation]{Theorem}
\newtheorem*{theorem*}{Theorem}
\newtheorem{lemma}[equation]{Lemma}
\newtheorem{proposition}[equation]{Proposition}
\newtheorem{corollary}[equation]{Corollary}

\theoremstyle{definition}

\newtheorem{definition}[equation]{Definition}

\newtheorem{property}[equation]{Property}

\theoremstyle{remark}
\newtheorem{remark}[equation]{Remark}

\numberwithin{equation}{section}

\newtheorem*{ack}{Acknowledgements}

\title{Triangulating surfaces with bounded energy}

\author{Maciej Borodzik}
\address{Institute of Mathematics, University of Warsaw, ul. Banacha 2,
02-097 Warsaw, Poland}
\email{mcboro@mimuw.edu.pl}
\author{Monika Szczepanowska}
\address{Institute of Mathematics, University of Warsaw, ul. Banacha 2,
02-097 Warsaw, Poland}
\email{monika.szczepanowska.waw@gmail.com}



\begin{document}
\maketitle

\begin{abstract}
  We show that if a closed $C^1$-smooth surface in a Riemannian manifold has bounded Kolasinski--Menger energy, then
  it can be triangulated with triangles whose number is bounded by the energy and the area. Each of the triangles is an image
  of a subset of a plane under a diffeomorphism whose distortion is bounded by $\sqrt{2}$.
\end{abstract}

\section{Introduction}

It is a general principle in the theory of energies of manifolds that
small energy implies uncomplicated topology. Probably, the first instance of
this principle is F\'ary--Milnor theorem \cite{Fary,Milnor}, stating that
a knot in $\R^3$ whose total curvature is less than $4\pi$ is necessarily trivial. 

For energies of curves in $\R^3$, there
are bounds for the stick number and the average crossing number
of a knot, see for example \cite{SSH} and references therein. 

For higher dimensional submanifolds some analogs exist, but are not abundant. The F\'ary--Milnor theorem can be generalized to 
the case of surfaces \cite{Hass}. Compactness results show that there
are finitely many isotopy classes of submanifolds below some
fixed energy level; see \cite{compiso}.

Motivated by \cite{Hass}, we give another bound on the complexity of a surface in $\R^n$ in terms of its energy. The topological complexity is
measured by the minimal number of triangles in the triangulation. In particular, we bound the genus of a surface in terms of its energy.
Actually, our result goes further.
For a surface with given energy we construct a triangulation in such a way that each
triangle is a graph of a function with bounded derivative and distortion. In this sense, the triangles in the triangulation are ``almost flat''.

Noting that the energy $\cE_p^{\ell}$
is introduced in~Definition~\ref{def:kolas_menger},
we present now 
the main result of this paper,

\begin{theorem}\label{thm:main}
  Suppose $\Sigma\subset\R^n$ is a closed surface. Let $\ell\in\{1,\dots,4\}$ and $p>2\ell$. Suppose $\Sigma$ has energy
  $\cE^{\ell}_p(\Sigma)=E<\infty$ and area $A$. Then $\Sigma$ can
  be triangulated with $C_2 A E^{2/(p-2\ell)}$ triangles,
  where $C_2$ is a universal constant depending only on $p,\ell$ and $n$.

  Each of the triangles is a graph of an open subset of a plane under a function whose derivative has norm bounded by $\sqrt{2}$
  and whose distortion is bounded by $\sqrt{2}$.
\end{theorem}
Combining Theorem~\ref{thm:main} with \cite[Theorem 1.1]{triang} stating that the minimal number $T(g)$ of triangles in a triangulation
of a closed surface of genus $g$ grows linearly with $g$, we obtain the following result.
\begin{corollary}
  There is a constant $C_g$ (depending on $\ell,n,p$) such that if $\Sigma\subset\R^n$ is a closed surface with $\cE^\ell_p(\Sigma)=E<\infty$ and area $A$, where $p>2\ell$,
  then $g(\Sigma)\le C_g A E^{2/(p-2\ell)}$.
\end{corollary}

The proof of Theorem~\ref{thm:main} goes along the following lines. 
The key tool is the Regularity Theorem of \cite{compiso}, recalled as Theorem~\ref{thm:regularity}, which states that an $m$-dimensional submanifold with bounded energy can
be covered by so-called graph patches, that is, subsets that are graphs of functions from subsets of $\R^m$ with bounded derivatives. The subsets in the cover have
diameter controlled by the energy, that is, they are not too small. An immediate corollary of Theorem~\ref{thm:regularity}
is an Ahlfors like inequality, Proposition~\ref{prop:local_volume_bound}, controlling from both sides the volume of the part of a submanifold 
cut out by a ball whose center is on the submanifold.

Next, assuming $\Sigma$ is a surface of bounded energy and area,
we find a cover of $\Sigma$ by balls of some radius $r$ (depending on the energy), such that each ball is a graph patch. We let
the centers of the balls be $x_1,\dots,x_N$. A simple topological argument in Subsection~\ref{sub:nets} allows us to control the number 
$N$ of the centers
in terms of the energy and area of the surface.

To construct the triangulation, we connect all pairs $x_i,x_j$, such that $||x_i-x_j||<4r$, by an arc $\gamma_{ij}$. The requirement
on $\gamma_{ij}$, spelled out in Definition~\ref{def:good_arc}, is that the length $\ell(\gamma_{ij})$ be bounded by 
a constant times 
$||x_i-x_j||$. Unlike geodesics, two such curves can intersect at more than a single point. By a procedure called \emph{bigon removal}
we improve the collection of curves $\gamma_{ij}$ to obtain a concrete bound on the number of intersection points between
them; see Lemma~\ref{lem:G2}.

We let $\Sigma_0$ be the complement of $\bigcup \gamma_{ij}$ in $\Sigma$.
The triangulation is constructed by cutting connected components of $\Sigma_0$
into triangles. 
A second technicality, and chronologically the first we deal with in this article, appears.  We need to 
ensure that each connected component of $\Sigma_0$ is planar. We address this problem
in Lemma~\ref{lem:G1}. Given that lemma, we consider each component $C$ of $\Sigma_0$. As it is planar, 
we cut $C$ into triangles without adding new vertices.
The number of triangles in the triangulation  is estimated using the number of intersection points between curves $\gamma_{ij}$;
see Corollary~\ref{cor:triang}. The proof of Theorem~\ref{thm:main} is summarized in Subsection~\ref{sub:main}.

\begin{ack}
  The paper is based on the Master thesis of the second named author under supervision of the first named author. The authors are grateful
  to S\l{}awomir Kolasi\'nski for fruitful discussions. 

  MB was supported by the Polish National Science Grant 2019/B/35/ST1/01120. MS was supported by the National Science Center grant 2016/22/E/ST1/00040.
\end{ack}

\section{Review of surface energies}
In this section we recall definitions of surface energy. References include \cite{geomsobol,compiso,mengersurf}.

\subsection{Discrete Menger energy for a submanifold in $\R^n$}

For points $(x_0,\dots,x_{m})$ in $\R^n$,
we let $\Delta(x_0,\dots,x_{m})$ denote the $m$-dimensional simplex spanned by $x_0,\dots,x_{m}$ (the convex hull of these points).
We define
\[\cK(x_0,\dots,x_{m})=\frac{1}{d^{m}}\vol_{m}\Delta(x_0,\dots,x_{m}),\]
where $d$ is the diameter of $\Delta(x_0,\dots,x_{m})$.

Suppose now $\Sigma\subset\R^n$ is a Lipshitz submanifold of dimension $m<n$. Let $\ell\in\{1,\dots,m+2\}$ and $p>0$.
\begin{definition}\label{def:kolas_menger}
  The \emph{Kolasinski-Menger energy} of $\Sigma$ is the integral
  \begin{equation}\label{eq:edef}\cE^\ell_p(\Sigma)=
    \int_{\Sigma^\ell} \sup_{x_\ell,\dots,x_{m+2}\in\Sigma}
    \cK(x_0,\dots,x_{m+2})^p\, d\!\vol(\Sigma^\ell).
  \end{equation}
  The integral is computed
  with respect to variables $x_0,\dots,x_{\ell-1}$.
\end{definition}

\subsection{Graph patches}\label{sub:graph_patches}

For $\alpha\in(0,1]$, we let $\cCamn$ denote the set of
all compact smooth manifolds of dimension $m$, embedded
$\Coa$-smoothly in $\R^n$. The following definition
is taken from \cite[Definition 1]{compiso}
\begin{definition}[Graph patches]
  Suppose $R,L,d$ are real positive and $\alpha\in(0,1]$.
  The class $\cCamn(R,L,d)$ is the class of all $m$-dimensional
  submanifolds $\Sigma\subset\R^n$ such that:
  \begin{enumerate}[label=(P-\arabic*)]
    \item $\Sigma\subset B(0,d)$ \label{item:diameter_bound};
    \item for each $x\in\Sigma$, there exists a function $f_x\colon
      T_x\Sigma\to T_x\Sigma^{\perp}$ of class $\Coa$ 
      with $f_x(0)=0$, $Df_x(0)=0$ and\label{item:graph_patch}
      \[\Sigma\cap B(x,R)=(x+\grf(f_x))\cap B(x,R).\]
    \item the function $f_x$ is Lipshitz with Lipshitz constant $1$
      and $||D f_x(\xi)-Df_x(\eta)||\le L|\xi-\eta|^\alpha$.\label{item:controled_bending}
  \end{enumerate}
\end{definition}

We quote now the following result 
\begin{theorem}[\expandafter{\cite[Regularity Theorem]{compiso}}]\label{thm:regularity}
  For $p>m\ell$, there exist constants $c_1(m,n,\ell,p)$
  and $c_2(m,n,\ell,p)$ such that with $\alpha=1-m\ell/p$, any
  Lipshitz manifold $\Sigma\in\cCamn$ with energy $\cE_p^\ell=E<\infty$
  satisfies
  \[\Sigma\in\cCamn(c_1 E^{-1/(p-m\ell)},c_2 E^{1/p},d)\]
  as long as $\Sigma\subset B(0,d)$.
\end{theorem}

We now introduce some notation regarding graph patches.
Let $x\in\Sigma$ and $r<c_1E^{-1/(p-m\ell)}$. 
We let $H\subset\R^n$ be the tangent plane to $\Sigma$ at~$x$. The map $f_x$ induces a map
$\phi_x\colon H\to\Sigma$ given by $y\mapsto(y,x+f_x(y))$, where we identify $H$ with $T_x\Sigma$. The inverse map
$\pi_x\colon\Sigma\to H$ is a projection along $H^\perp$. Both maps $\phi_x$ and $\pi_x$ are defined only in a neighborhood of $x$.

We will use the following corollary of Theorem~\ref{thm:regularity}. We use the notation $B^H$ for a ball contained in $H$, the notation
like $B(y,\rho)$ means an open ball in the ambient space $\R^n$.
\begin{corollary}\label{cor:regularity}
  Let 
  \begin{equation}\label{eq:r0def}
    R_0=2^{-1/2}\min(c_1,c_2^{-1/p\alpha})E^{-1/(p-m\ell)}.
  \end{equation}
  If $r<R_0$, then:
  \begin{enumerate}[label=(C-\arabic*)]
    \item $\phi_x$ is well-defined on $H\cap B^{H}(x,r)$;\label{item:phi}
    \item $\pi_x$ is well-defined on $\Sigma\cap B(x,r\sqrt{2})$;\label{item:pi}
    \item for any $s<r\sqrt{2}$, the image $B^H(x,s)\subset \pi_x(\Sigma\cap B(x,s))\subset B^H(x,s\sqrt{2})$.\label{item:U}
    \item $||D\phi_x(y)||\le \sqrt{2}$ for all $y\in U$;\label{item:dphi}
    \item $\phi_x$ is Lipshitz with Lipshitz constant $\sqrt{2}$ and $\pi_x$ is Lipshitz with Lipshitz constant $1$.\label{item:lip}
  \end{enumerate}
\end{corollary}
\begin{proof}
  Property \ref{item:dphi} follows from \ref{item:controled_bending} and the fact that $Df_x(0)=0$. The map $\pi_x$ is a projection,
  so it is Lipshitz with Lipshitz constant $1$. Item~\ref{item:lip} is immediately deduced from \ref{item:controled_bending} by the triangle inequality.   Property \ref{item:pi}
  is just \ref{item:graph_patch} and property \ref{item:U} follows readily. \ref{item:phi} follows from \ref{item:U}.
\end{proof}
\begin{corollary}\label{cor:distort}
  The distortion of~$\phi_x$ is bounded by $\sqrt{2}$.
\end{corollary}
\begin{proof}
  The distortion of $\phi_x$ at the point $z$ is given by 
  \[D(z)=\lim\sup_{r\to 0}\dfrac{\max_{y\colon ||x-y||=r}||\phi_x(y)-\phi_x(z)||}{\min_{y\colon ||x-y||=r}||\phi_x(y)-\phi_x(z)||}.\]
  The numerator in the formula is bounded from above by $r$ times the Lipshitz constant of $\phi_x$. The denominator
  is bounded from below by $r$ times the Lipshitz constant of $\pi_x$.
\end{proof}
  
\subsection{Local volume bound}\label{sub:applications}
Throughout Subsection~\ref{sub:applications} we let $\Sigma$
be a submanifold of $\R^n$ in the class $\cCamn(R_0,L,d)$.

\begin{proposition}[Local volume bound]\label{prop:local_volume_bound}
Suppose $r<R_0$. Then for any $x\in\Sigma$ we have
\[2^{-n/2}V_mr^m < \vol(\Sigma\cap B(x,r)) < 2^{n-1}V_mr^m,\]
where $V_m$ is the volume of the unit ball in dimension $m$.
\end{proposition}
\begin{proof}
  Let $U=\phi_x(\Sigma\cap B(x,r/\sqrt{2}))$ as in \ref{item:U}. We know that 
  \begin{equation}\label{eq:U_again} 
    B^H(x,r/\sqrt{2})\subset U\subset B^H(x,r).
  \end{equation}
  As $\Sigma\cap B(x,r)$ is a graph of $\phi_x$, a classical result from multivariable calculus computes
  the volume of $\Sigma\cap B(x,r)$ in terms of integral over $U$ over the square root of the Gram determinant:
  \[\vol(\Sigma\cap B(x,r))=\int_U \sqrt{|\det G|},\]
  where $G=D\phi_x D\phi_x^T$.

  The derivative of $\phi_x$ has a block structure
  $D\phi_x=\begin{pmatrix} I & Df_x \end{pmatrix}$, so
  $G:=I+Df_xDf_x^T$. 
  On the one hand,
  since $Df_xDf_x^T$ is non-negative definite, $\det G\ge 1$. On the other hand,
  $||Df_x||<1$, so $||Df_xDf_x^T||<1$. Therefore, $||G||<2$. This means that all
  the eigenvalues of the symmetric matrix $G$ have modulus less than $2$, so $|\det G|<2^n$. 
  In particular,
  \[\int_U 1\le \vol(\Sigma\cap B(x,r))\le \int_U 2^{n-1}.\]
  Combining this with \eqref{eq:U_again}
  we quickly conclude the proof.
  %We have
  %\begin{multline*}
  %  2^{-n/2}V_mr^m=\vol(B^m(0,r/\sqrt{2}))\le \int_{B(0,r/\sqrt{2})}\sqrt{g}\le \\
  %\int_{(I,f_x)^{-1}(B(x,r))} \sqrt{g}\le \int_{B(0,r)} \sqrt{g}\le \int_{B(0,r)} 2^{n/2}V_n^{1/2}\le
  %2^{n/2}V_n^{1/2} V_mr^m.
%\end{multline*} 
\end{proof}
\section{Geodesic-like systems of curves}
\subsection{Nets of points}\label{sub:nets}
We will need the following technical definition.
\begin{definition}
  Let $r>0$. A finite set $\cX$ of points in $\Sigma$ is a \emph{$r$-net} if
  \begin{itemize}
    \item the balls $B(x,r)$ for $x\in\cX$ cover $\Sigma$;
    \item for any $x,x'\in\cX$, $x\neq x'$, we have $\dist(x,x')\ge r/2$.
  \end{itemize}
\end{definition}
The following result is classical in general topology.
\begin{proposition}
  Each compact submanifold $\Sigma$ admits an $r$-net.
\end{proposition}
\begin{proof}
  For the reader's convenience we provide a quick proof.
  Cover first $\Sigma$ by all balls $B(x,r/2)$ with $x\in\Sigma$. Choose a finite subcover, and let $\cY=\{y_1,\dots,y_M\}$ be the
  set of centers of balls in this subcover.

  We act inductively. Start with $y_1$. Remove from $\cY$ all points $y_i\neq y_1$ such that
  $\dist(y_1,y_j)<r/2$. After this procedure, the balls $B(y_1,r)$ and $B(y_i,r/2)$ for $i>1$ still cover $\Sigma$.

  For the inductive step assume that, for given $n$,
  the balls $B(y_1,r),\dots,B(y_{n-1},r)$ and  $B(y_n,r/2)$, $B(y_{n+1},r/2),\dots$, $B(y_M,r/2)$
  cover $\Sigma$, and 
  there are no indices $i,j$ with $i<n$ and $j\ge n$ such
  that $\dist(y_i,y_j)<r/2$. We remove from $\cY$ points $y_j$ with $j>n$ such that $\dist(y_n,y_j)<r/2$.

  After a finite number of steps we are left with the set $\cX\subset\cY$, such that
  $B(x_i,r)$, $x_i\in\cX$ cover $\Sigma$ and $\dist(x_i,x_j)\ge r/2$ for all $x_i,x_j\in\cX$.
\end{proof}
In case $r<R_0$ we can bound the number of elements in the net via the following.
\begin{proposition}[Bounding $|\cX|$]\label{prop:covering}
  Suppose $\Sigma\in\cCamn$ is such that $E=\cE_p^\ell(\Sigma)<\infty$.
  Let $A=\vol_m(\Sigma)$ and $R_0$ be given by \eqref{eq:r0def}.
  If $r<R_0$, then any $r$-net $\cX$ has
  $|\cX|<2^{n/2+4m}V_m^{-1}r^{-m}A$.
\end{proposition}
\begin{proof}
  By the local volume bound (Proposition~\ref{prop:local_volume_bound})
  the balls $B(x_i,r/4)\cap\Sigma$, $x_i\in\cX$ have volume at least $2^{-n/2-4m}V_m r^m$, and are pairwise
  disjoint. So the total volume of $\bigcup_{x_i\in\cX} B(x_i,r/4)\cap\Sigma$ is at least
  $2^{-n/2-4m}V_mr^m|\cX|$. This quantity does not exceed the volume of $\Sigma$.
\end{proof}
Essentially the same argument yields the following result.
\begin{proposition}\label{prop:sigma_bound}
  Let $\Sigma$ be as in Proposition~\ref{prop:covering}.
  Suppose $\sigma>0,r>0$ are such that $(\sigma+1/4) r<R_0$. Let $\cX$ be an $r$-net. Each
  ball $B(x,\sigma r)$ for $x\in\Sigma$ contains at most $\cT_m(\sigma)$ points from $\cX$, where
  \[\cT_m(\sigma)=2^{3n/2+4m}(\sigma+1/4)^m.\]
\end{proposition}
\begin{proof}
  The ball $B(x,(\sigma+1/4)r)\cap\Sigma$ has volume at most $2^{n-1}V_m(\sigma+1/4)^mr^m$. All balls of radius
  $r/4$ with centers at $x_i\in\cX$ such that $\dist(x,x_i)<\sigma r$ are pairwise disjoint,
  belong to $B(x,(\sigma+1/4)r)$, and have volue at least $2^{-n/2-4m}V_mr^m$. Hence,
  the number of points in $\cX$ at distance at most $\sigma r$ to $x$ is
  bounded by $2^{3n/2+4m-1}(\sigma+1/4)^m$.
\end{proof}
From now on, we set
\[\cT(\sigma)=\cT_2(\sigma)=2^{3n/2+7}(\sigma+1/4)^2.\]

\subsection{Good arcs}
%Throughout Section~\ref{sec:graph_patches}
%we let $\Sigma$ be a surface with area $A$ and energy $E$. We also let
%$r>0$ be such that $22\frac14r<R_0$.

\begin{definition}\label{def:good_arc}
  Let $\cX=\{x_1,\dots,x_N\}$ be an $r$-net.
  Let $\cI$ be a subset of pairs $(i,j)$ with $1\le i<j\le N$. 
  A collection $\cG$ of arcs $\gamma_{ij}$, $(i,j)\in\cI$ smoothly embedded in $\Sigma$ and connecting $x_i$
  with $x_j$ is called a collection of \emph{good arcs associated to $\cX$} if:
  \begin{itemize}
    \item $\cI=\{(i,j)\colon i\neq j,\ ||x_i-x_j||<4r\}$;
    \item for all $(i,j)\in\cI$, $\gamma_{ij}$ has length less than $2\dist(x_i,x_j)$;
    \item $\gamma_{ij}=\gamma_{ji}$;
    \item if $\gamma_{ij}\neq \gamma_{k\ell}$, then $\gamma_{ij}$ is transverse to $\gamma_{k\ell}$.
  \end{itemize}
  A collection of good arcs is called \emph{tame}, if it additionally satisfies the following two conditions.
  \begin{enumerate}[label=(G-\arabic*)]
    \item Every connected component of $\Sigma\setminus \Gamma$ with $\Gamma=\bigcup \gamma_{ij}$ is homeomorphic to an open set of $\R^2$;\label{item:G1}
    \item Unless $\gamma_{ij}=\gamma_{k\ell}$, the curves $\gamma_{ij}$ and $\gamma_{k\ell}$ intersect transversally at at most $\cT(17)$ points.\label{item:G2}
  \end{enumerate}
\end{definition}
One should think of a collection of good arcs as an analog of a collection of geodesics. Condition~\ref{item:G2} is automatically
satisfied if $\gamma_{ij}$ are geodesics whose length is less than the geodesic radius.

\begin{proposition}\label{prop:good_arcs}
  % add about r and R_0
  Suppose $4r<R_0$.
  For each $r$-net $\cX$ there exists an associated collection of good arcs.
\end{proposition}
\begin{proof}
  Take two points $x_i,x_j\in\cX$ with $\dist(x_i,x_j)<4r$. We want to show that there
  exists a curve $\gamma_{ij}$ on $\Sigma$ connecting them with $\ell(\gamma_{ij})<2\dist(x_i,x_j)$.
  To see this, let $y_j=\pi_{x_i}(x_j)$, where $\pi_{x_i}$ is the projection onto $H_i$, the tangent plane to $x_i$; compare Corollary~\ref{cor:regularity}. As $4r<R_0$, by \ref{item:pi}, $\pi_{x_i}$ is defined on $x_j$. Moreover, by \ref{item:phi}, 
  the segment on $H_i$ connecting $x_i$ and $y_j$
  belongs to the domain of $\phi_{x_i}$.

  Let $\rho$ be the segment on $H_i$ connecting $x_i$ and $y_j$. As $\pi_{x_i}$ has Lipshitz constant $1$ by \ref{item:pi},
  we have that $||x_i-y_j||\le\dist(x_i,x_j)$, so the length of $\rho$ is less than or equal to $\dist(x_i,x_j)$.
  By \ref{item:dphi} we infer that the length of $\gamma_{ij}:=\phi_{x_i}(\rho)$ is at most 
  $\sqrt{2}\dist(x_i,x_j)<2\dist(x_i,x_j)$.

  We construct all the curves in $\cG$ one by one, making a newly constructed curve transverse to the previous ones. Note that transversality of $\gamma_{ij}$ to some $\gamma_{k\ell}$ is equivalent to transversality of $\rho=\pi_{x_i}(\gamma_{ij})$ to $\pi_{x_i}(\gamma_{k\ell})$. Therefore, we can always perturb $\rho$, which is a planar curve, to be transverse to all
  previously constructed curves by standard transversality arguments.
\end{proof}
Our goal is to show that there exists a tame collection of good arcs. In Subsection~\ref{sub:G1} we shall
deal with Condition~\ref{item:G1}, while in Subsection~\ref{sub:G2} we deal with \ref{item:G2}.
\subsection{On the property~\ref{item:G1}}\label{sub:G1}
\begin{lemma}\label{lem:G1}
  If $22\frac14r<R_0$, a collection of good arcs $\cG$
    satisfies \ref{item:G1}.
\end{lemma}
\begin{proof}
  \begin{figure}
    \includegraphics[width=0.9\linewidth]{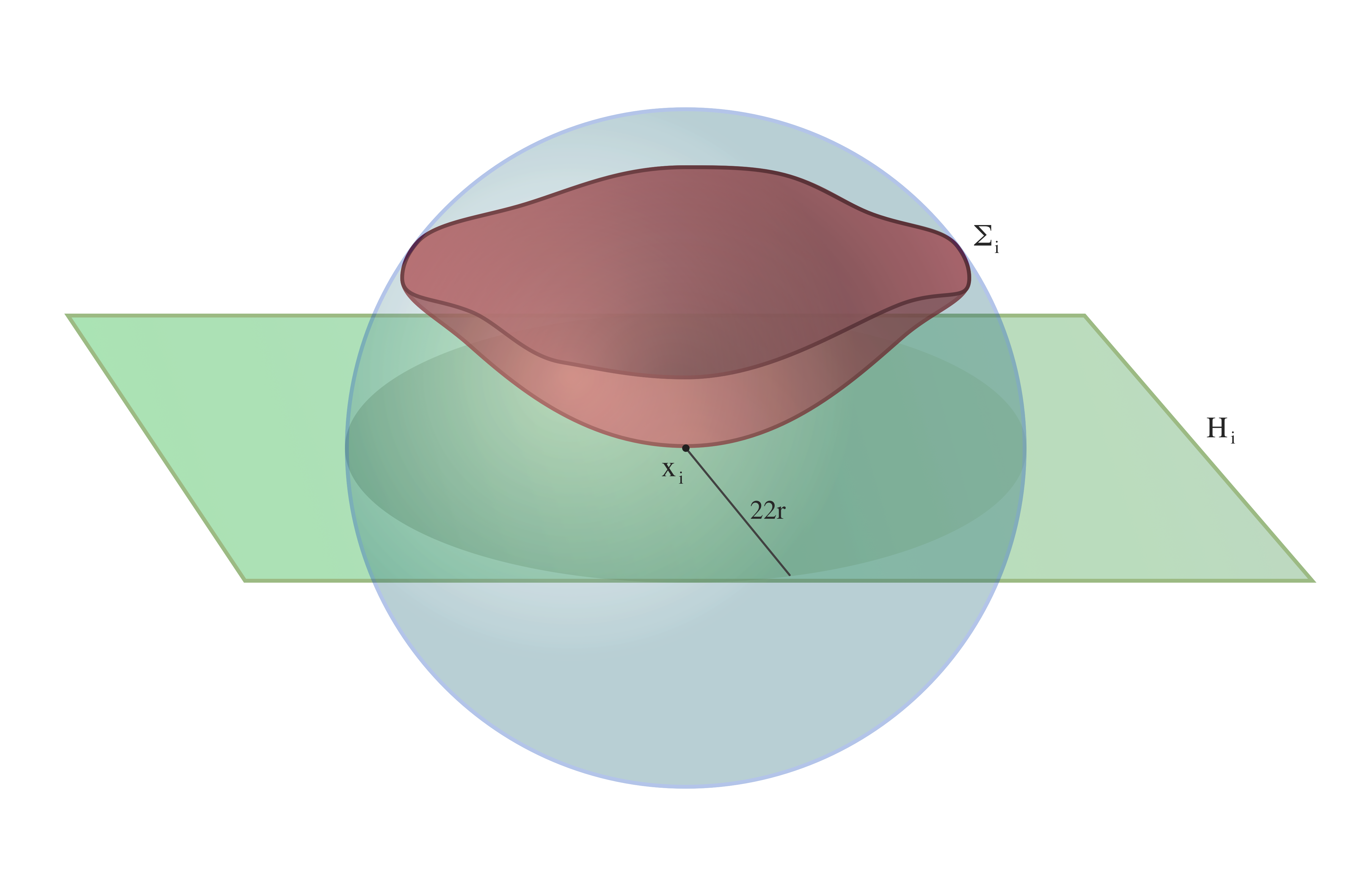}
    \caption{The set $H_i$ and $\Sigma_i$ of the proof of Lemma~\ref{lem:G1}.}\label{fig:G10}
  \end{figure}
  We use the notation of Corollary~\ref{cor:regularity}; compare Figure~\ref{fig:G10}. Let $x_i\in\cX$.
  Let $H_i$ be the plane tangent to $\Sigma$ at $x_i$. Let $\Sigma_i=\Sigma\cap B(x,22r)$. As $22r<R_0$, properties \ref{item:phi}--\ref{item:lip} are satisfied. %The maps $\phi_{x_i}$ and $\pi_{x_i}$ are denoted $\}phi_i$ and $\pi_i$.
  We let $U_i=\pi_{x_i}(\Sigma\cap B(x_i,22r))$. We have
  \[B^{H_i}(x_i,15r)\subset B^{H_i}(x_i,22r/\sqrt{2})\subset U_i\subset B^{H_i}(x_i,22r).\]

  Consider a regular $38$-gon on $H_i$ with center $x_i$ and side length $r$. Denote by $u_1,\dots,u_{38}$ its vertices, so that
  $||u_j-u_{j+1}||_{H_i}=r$. We have $||u_j-x_i||_{H_i}=\frac{r}{2\sin\frac{\pi}{38}}\sim 6.05r$. In particular
  \begin{equation}\label{eq:ujxi}
    6r<||u_j-x_i||_{H_i}<7r.
  \end{equation}
  \begin{figure}
    \includegraphics[width=0.5\linewidth]{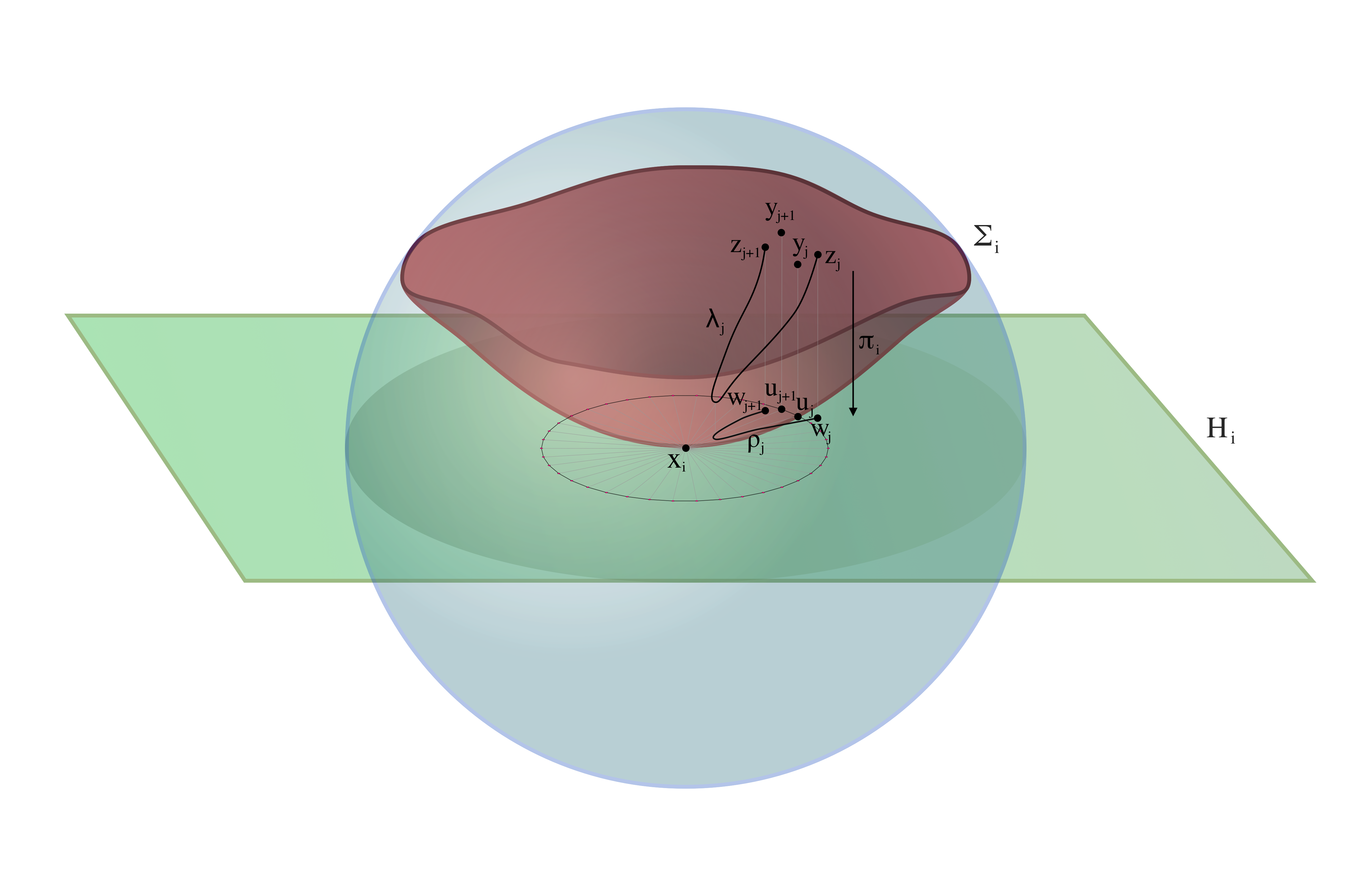}\hskip 1cm
    \includegraphics[width=0.3\linewidth]{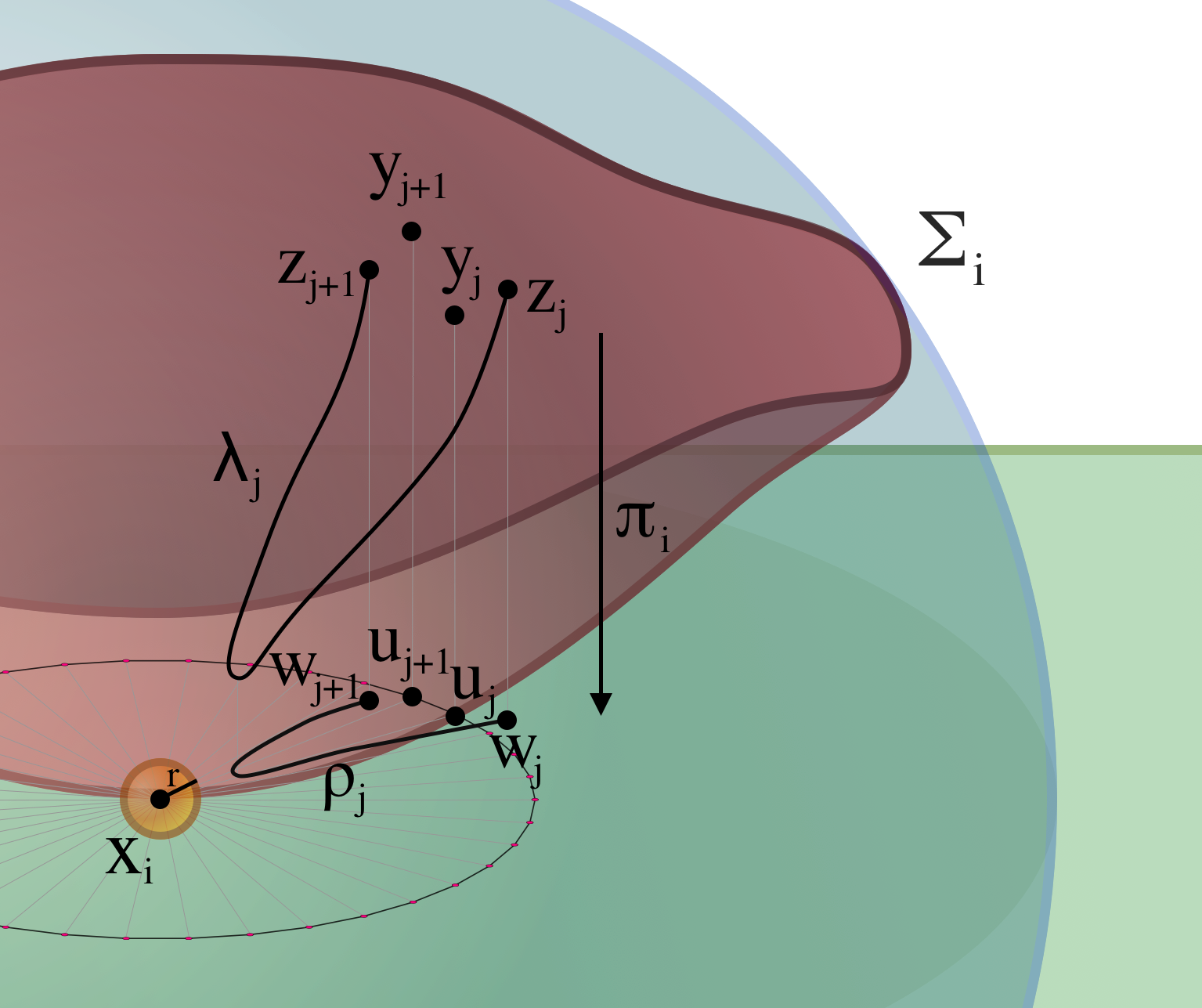}
    \caption{Notation of Subsection~\ref{sub:G1}.}\label{fig:G1}
  \end{figure}
  Let $y_j=\phi_{x_i}(u_j)\in\Sigma_i$. By \ref{item:lip}:
  \begin{equation}\label{eq:yjyjp}
    r<\dist(y_j,y_{j+1})<r\sqrt{2}.
  \end{equation}
  By the definition of $\cX$, for any $y_j$ there exists an element $z_j\in\cX$
  such that $\dist(z_j,y_j)<r$. In particular, by the triangle inequality and \eqref{eq:yjyjp}:
  \begin{equation}\label{eq:zjzjp}
    \dist(z_j,z_{j+1})<(2+\sqrt{2})r,\ \ \dist(x_i,y_j)<7r\sqrt{2}<10r.
  \end{equation}
  As $\cG$ is a collection of good arcs, and $\dist(z_j,z_{j+1})<4r$, there exists a curve $\lambda_j\in\cG$ that connects $z_j$ and $z_{j+1}$.
  The length $\lambda_j$ is at most $2\dist(z_j,z_{j+1})<(4+2\sqrt{2})r<7r$.

  Denote $w_j=\pi_{x_i}(z_j)$ and let $\rho_j=\pi_{x_i}(\lambda_j)$. 
  %\todoMB{Another picture with curves $\rho_j$}
  \begin{lemma}\label{lem:belongs_to}
    We have $z_j\in B(x_i,11r)$. Moreover, $\lambda_j\subset B(x_i,15r)$  and
    $\rho_j$ belongs to $B^{H_i}(x_i,15r)$.
  \end{lemma}
  \begin{proof}
    From \eqref{eq:ujxi} we read off that $\dist(x_i,y_j)<10r$. As $\dist(y_j,z_j)<r$, we conclude that $\dist(x_i,z_j)<11r$.
%and the fact that $\phi_i$ is has Lipshitz constant $\sqrt{2}$ we infer that $||x_i-\phi_i(u_j)||_{H_i}<7r\sqrt{2}<10r$.

    The curve $\lambda_j$ has length at most $7r$. No point on $\lambda_j$ can be further from $x_i$ than $11r+\frac72r<15r$. Indeed, if $x$ is outside $B(x_i,15r)$, then the length of the part of $\lambda_j$ from $z_j$ to $x$ and the length
    of the part of $\lambda_j$ from $x$ to $z_{j+1}$ are both at least $4r$, contributing to the length of $\lambda_j$ being at least $8r$.
Therefore,
    $\lambda_j\subset B(x_i,15r)$. 

    Now $\pi_{x_i}(B(x_i,15r))\subset B^{H_i}(x_i,15r)$, so $\rho_j\subset B^{H_i}(x_i,15r)$.
  \end{proof}

  %As $\pi_i$ is a Lipshitz map with Lipshitz constant $1$, the length
  %of $\rho_j$ is at most $7r$. 
  As $\dist(z_j,y_j)<r$, we have $||w_j-u_j||<r$. By \eqref{eq:ujxi} and the
  triangle inequality:
  \begin{equation}\label{eq:wjxi}
    5r<||w_j-x_i||<8r.
  \end{equation}
  \begin{lemma}\label{lem:misses}
    Let $H_{ij}^+$ be the half-plane cut off from $H_i$ by the line parallel to the segment joining $u_j$ and $u_{j+1}$ and passing through
    $x_j$ such that $u_j,u_{j+1}\in H_{ij}^+$; see Figure~\ref{fig:misses}.
    The curve $\rho_j$ misses $B^{H_i}(x_i,r)$ and $H_{ij}^+$. In particular, $\lambda_j$ is disjoint from $B(x_i,r)$.
  \end{lemma}
  \begin{figure}
    \includegraphics[width=0.4\linewidth]{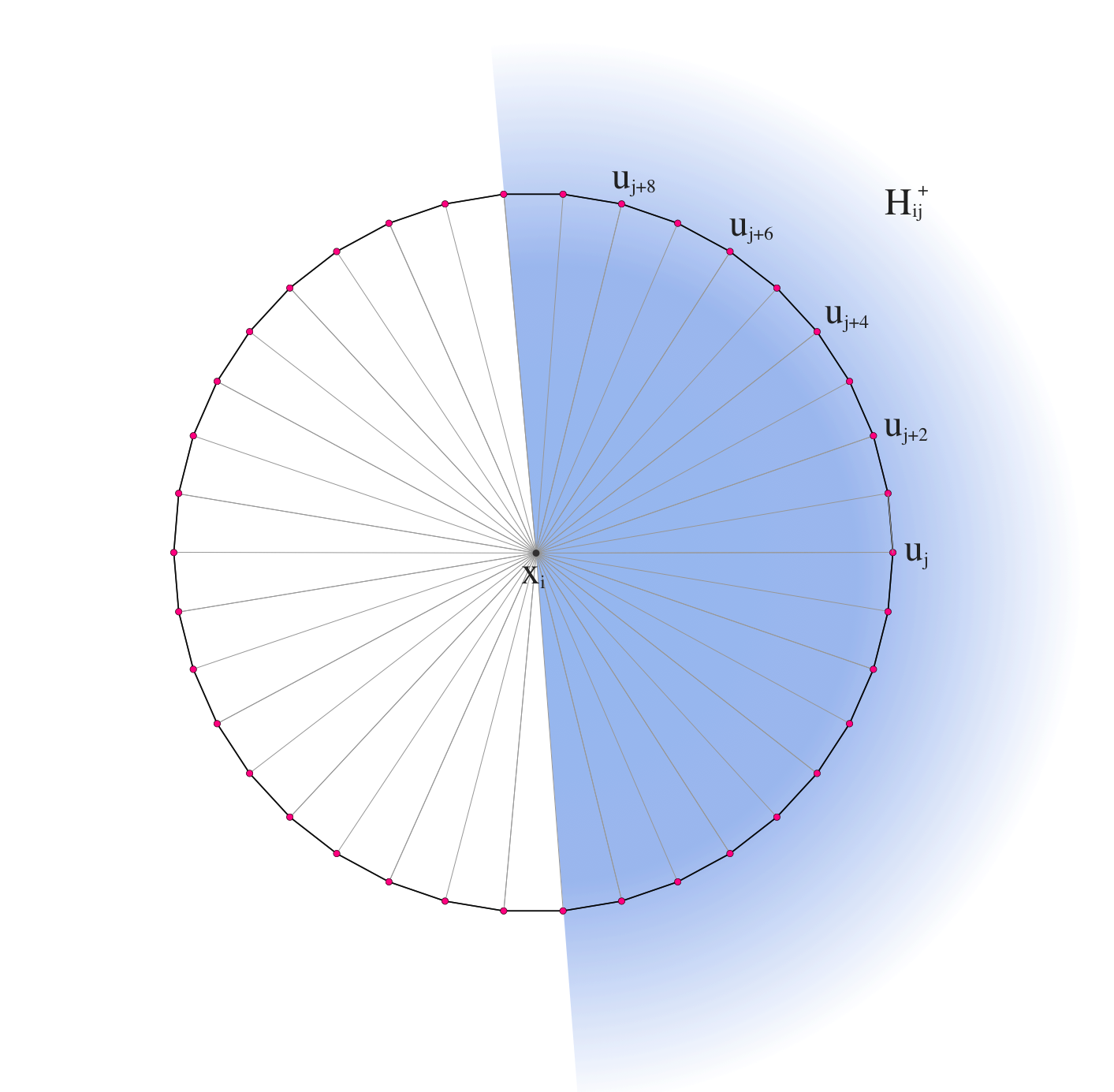}\hskip1cm\includegraphics[width=0.4\linewidth]{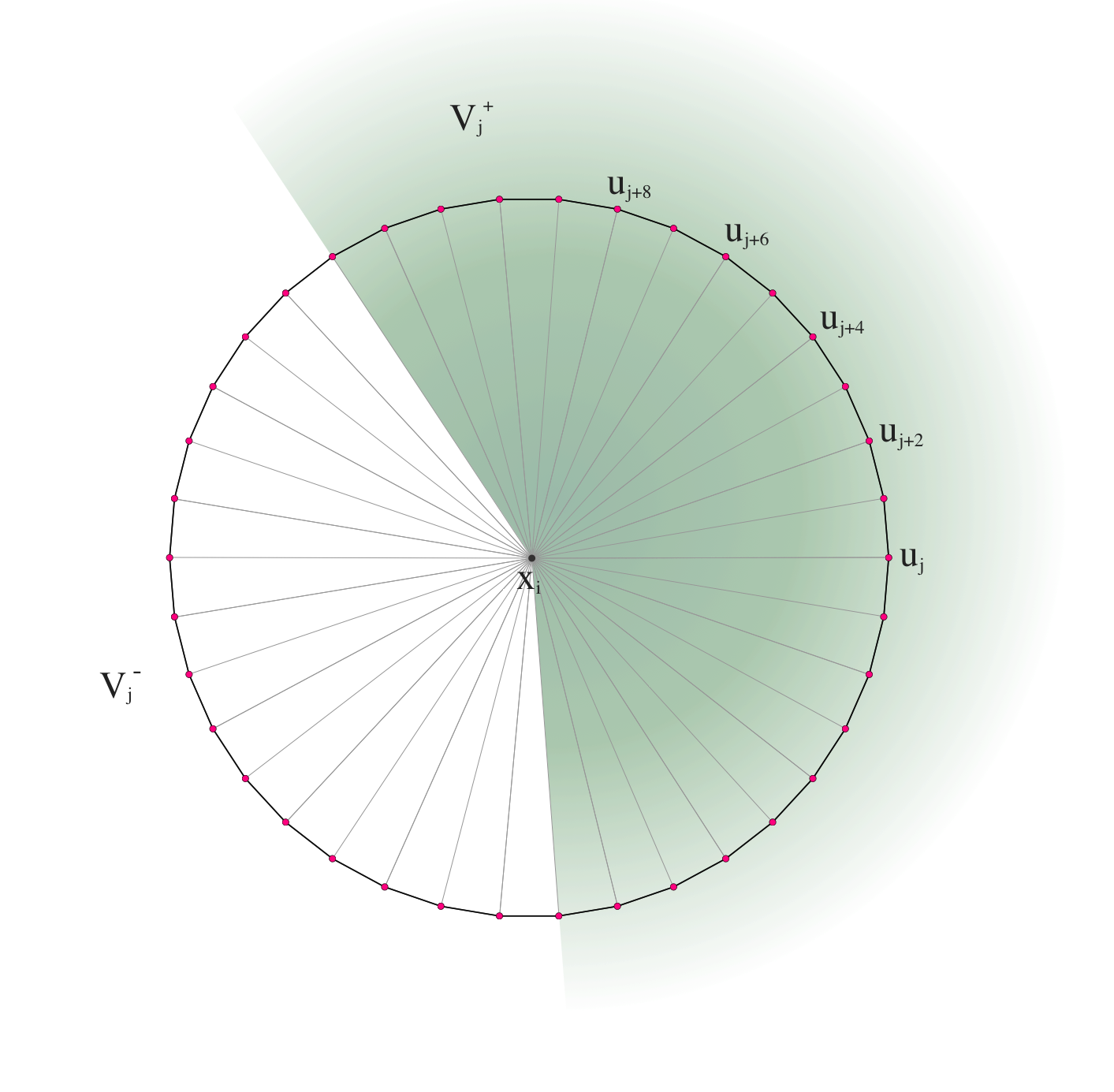}
    \caption{The set $H_{ij}^+$ (left) and $V_{j}^+$ (right) of the proof of Lemma~\ref{lem:misses} and~\ref{lem:positive_increment}.}\label{fig:misses}
  \end{figure}
  \begin{proof}[Proof of Lemma~\ref{lem:misses}]
    Note that $\pi_{x_i}$ being Lipshitz with Lipshitz constant $1$ implies that the length of $\rho_j$ is at most $7r$; see \ref{item:lip}.
    Suppose towards contradiction that $\rho_j$ passes through a point $z\in B(x_i,r)$. We have $||w_j-z||>4r$ and $||w_{j+1}-z||>4r$
    by \eqref{eq:wjxi} and the triangle inequality, hence the length of $\rho_j$ is at least $8r$. The contradiction shows
    that $\rho_j$ misses the ball $B^{H_i}(x_i,r)$. 

    Suppose now $\lambda_j$ hits the ball $B(x_i,r)$, that is, there exists a point $x\in\lambda_j\cap B(x_i,r)$. Then,
    $\pi_{x_i}(x)\in B^{H_i}(x_i,r)$ by \ref{item:lip} and so $\rho_j$ hits $B^{H_i}(x_i,r)$, contradicting what we have already proved. This shows
    that $\lambda_j$ is disjoint from $B(x_i,r)$.

    To prove that $\rho_j$ misses $H_{ij}^+$ is analogous. The distance of $u_j$ to the boundary of $H_{ij}^+$ is equal to 
    $r\sqrt{1/4\sin^2(\pi/38)-1/4}$,
    and it is greater than $6r$. Hence, the distance of $w_j$ to $\partial H_{ij}^+$ is at least $5r$. That is, if $\rho_j$
    leaves $H_{ij}^+$, then its length must be at least $10r$. Contradiction.
  \end{proof}
  Continuing the proof of Lemma~\ref{lem:G1}, we prove the next result.
  \begin{lemma}\label{lem:positive_increment}
    For any $j$, the increment of the argument along $\rho_j\cup\rho_{j+1}\cup\dots\cup\rho_{j+s}$ with $s=3,4$ is positive.
  \end{lemma}
  \begin{proof}
    The proof works for larger $s$, we only need the case $3$ and $4$. We prove for $s=3$, leaving the analogous case $s=4$ to the reader.
    Define: 
    $\ol{\rho}=\rho_j\cup\rho_{j+1}\cup\rho_{j+2}$. % see Figure~\ref{fig:olrho}.
    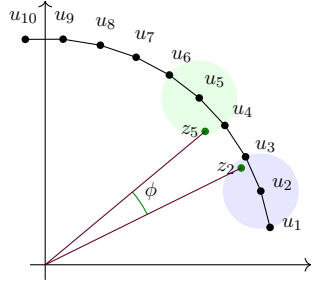
\begin{figure}
      \begin{tikzpicture}
	\draw[->] (-0.2,0) -- (3.5,0);
	\draw[->] (0,-0.2) -- (0,3.5);
	\foreach \i in {1,...,10}
	{
	  \coordinate (A\i) at (9.5*\i:3);
	}

	\fill[blue!10] (A2) circle (0.5);
	\fill[green!10] (A5) circle (0.5);

	\foreach \i in {1,...,10}
	{
	  \fill[black] (A\i) circle (0.05);
	  \draw (A\i) ++ (9.5*\i:0.3) node [scale=0.7] {$u_{\i}$};
	}
	\draw (A1) -- (A2) -- (A3) -- (A4) -- (A5) -- (A6) -- (A7) -- (A8) -- (A9) -- (A10);
	%\draw[thin] (0,0) -- (A2);
	%\draw[thin] (0,0) -- (A5);
	%\draw (0,0) ++ (18.9:0.9) arc [start angle=18.9, delta angle=27.9, radius=0.9];
	%\draw (0,0) ++ (32:1.2) node [scale=0.7] {$\frac{3\pi}{19}$};
	\draw (A2) ++ (130:0.4) coordinate (B);
	\draw (A5) ++ (280:0.45) coordinate (C);
	\begin{scope}
	  \clip (0,0) -- (B) -- (C) -- (0,0);
	\draw[green!50!black] (0,0) circle (1.5);
	\draw (0,0) ++ (35:1.7) node [scale=0.7] {$\phi$};
	\end{scope}
	\fill[green!50!black] (B) circle (0.05);
	\fill[green!50!black] (C) circle (0.05);
	\draw (B) ++ (-0.2,0) node [scale=0.7] {$z_2$};
	\draw (C) ++ (-0.2,0) node [scale=0.7] {$z_5$};
	\draw[thin,purple!50!black] (0,0) -- (B);
	\draw[thin,purple!50!black] (0,0) -- (C);
      \end{tikzpicture}
      \caption{Proof of Lemma~\ref{lem:positive_increment}. The points $z_2$ and $z_5$ belong to the disks with centers at
      $u_2$, respectively $u_5$. No matter where the points $z_2$ and $z_5$ are, the oriented angle $\phi$ is positive.}\label{fig:olrho}
    \end{figure}
    A geometric argument reveals that the oriented angle between the lines $\ol{x_iw_j}$ and $\ol{x_iw_{j+3}}$ is positive;
    compare Figure~\ref{fig:olrho}.
    Therefore, it is enough to prove that $\ol{\rho}$ does not make a full negative turn while going from $w_j$ to $w_{j+3}$.

    To this end, let $V^+_j=H_{ij}^+\cup H_{i,j+1}^+\cup H_{i,j+2}^+\cup H_{i,j+3}^+$ and $V^-_j=H_i\setminus V^+_j$. Then $V^-_j$ is non-empty $\ol{\rho}$
    misses $V^-_j$. That is, the curve $\ol{\rho}$ does not go around the point $x_j$ and the increment of the angle
    along $\ol{\rho}$ is the same as the oriented angle between the corresponding lines.
  \end{proof}
  \begin{corollary}\label{cor:delta_def}
    Set $\Delta=\rho_1\cup\dots\cup\rho_{38}$. The increment of the argument along $\Delta$ is a positive multiple of $2\pi$.
  \end{corollary}
  From this it follows that the winding index of $\Delta$ around 
  $x_i$ along is positive. %An argument involving Jordan curve theorem shows the following
  \begin{remark}
    With a little more care, we could prove that the winding
    index of $\Delta$ around $x_i$ is $1$. We do not need this result.
  \end{remark}
  \begin{corollary}\label{cor:jordan}
    The cycle $\Delta$ cuts $H_i$ into a finite number of components and $B^{H_i}(x_i,r)$ belongs to a bounded connected component of $H_i\setminus \Delta$.
  \end{corollary}
  \begin{proof}[Proof of Corollary~\ref{cor:jordan}]
    The cycle $\Delta$ is closed and has finitely many self-intersections. By Jordan curve theorem, $H_i\setminus\Delta$ consists of a finite number of bounded connected components and one unbounded connected component. As the winding index of $\Delta$ around $x_i$ is positive, 
    $x_i$ cannot possibly
    belong to the unbounded connected component. By Lemma~\ref{lem:misses}, $\Delta$
    is disjoint from $B^{H_i}(x_i,r)$. Hence, the whole of
    $B^{H_i}(x_i,r)$ belongs to the same connected component of $H_i\setminus\Delta$.
  \end{proof}
  We denote by $S_i\subset\Sigma$ the connected component of $\phi_{x_i}(H_i\setminus\Delta)$ containing $B(x_i,r)\cap\Sigma$. To conclude the proof of Lemma~\ref{lem:G1}, we choose a point $x\in\Sigma\setminus \Gamma$, where
  we recall that $\Gamma=\bigcup \gamma_{ij}$.
  By the definition of $\cX$, there exists a point $x_i\in\cX$ such that $x\in B(x_i,r)$. The connected component of $\Sigma\setminus\Gamma$
  containing $x$ is a subset of $S_i$, which is homeomorphic, via $\pi_{x_i}$ to an open subset of $H_i$. This concludes the proof.
\end{proof}
For future use we note the following corollary of the proof
of Lemma~\ref{lem:G1}.
\begin{corollary}\label{cor:for_future_use}
  Any connected component of $\Sigma\setminus\Gamma$ belongs
  to $B(x_i,22r)$ for some $x_i\in\cX$.
\end{corollary}
\begin{proof}
  With the notation of the proof of Lemma~\ref{lem:G1},
  we let $S$ be a connected component of $\Sigma\setminus\Gamma$.
  Take $x\in S$ and let $x_i\in\cX$ be such that $\dist(x_i,x)<r$.
  Then $x\in \phi_{x_i}(S_i)$ and then, as $\phi_{x_i}(S_i)$ is a connected component of
  $\Sigma\setminus(\lambda_1\cup\dots\cup \lambda_{38})\subset\Sigma\setminus\Gamma$, we clearly have $S\subset \phi_{x_i}(S_i)$. 
  Now $S_i\subset B^{H_i}(x_i,15r)$, so $\phi_{x_i}(S_i)\subset B(x_i,15\sqrt{2}r)\subset B(x_i,22r)$. Hence, $S\subset B(x_i,22r)$ as desired.
\end{proof}
\subsection{The property~\ref{item:G2}}\label{sub:G2}
We will now show that under the same conditions as in Lemma~\ref{lem:G1} we can improve the collection $\cG$ in such
a way that \ref{item:G2} is satisfied.
\begin{lemma}\label{lem:G2}
  Suppose $\cG$ is a collection of good arcs.
  If $29r<R_0$,
  then there exists a collection of good arcs $\wt{\gamma}_{ij}$ satisfying \ref{item:G2}.
\end{lemma}
\begin{proof}
  Choose $\varepsilon>0$ such that for any $i,j$ we have $\ell(\gamma_{ij})+\varepsilon< 2||x_i-x_j||$. We introduce the following notion.
  \begin{definition}
    A \emph{boundary bigon} is a pair of two arcs $\alpha$ and $\beta$ with common end points and disjoint interiors
    such that $\alpha$ and $\beta$ are parts of some curves $\gamma_{ij}$ and $\gamma_{kl}$.

    A \emph{bigon} $(D,\alpha,\beta)$ is a triple $(D,\alpha,\beta)$, where $(\alpha,\beta)$ form a boundary bigon
    and $D$ is an embedded disk
    $D\hookrightarrow \Sigma$ such that $\partial D=\alpha\cup\beta$ and $D$ belongs to $B(x_i,17r)$ for some $x_i$, which is an end point
    of a curve in $\cG$ whose part is either $\alpha$ or $\beta$.
  \end{definition}
  Suppose $\gamma_{ij}$ and $\gamma_{kl}$ are distinct arcs. Let $s$
  be the number of their intersection points. Then, $\gamma_{ij}$ and $\gamma_{kl}$ form $s-1$ boundary bigons.
  %Any pair of curves $\gamma_{ij}$ and $\gamma_{kl}$ that intersect at more than a single point forms at least one boundary bigon. 
  The
  proof of Lemma~\ref{lem:G2} relies on successively removing boundary bigons.
  \begin{lemma}\label{lem:bigon_is_bigon}
    Every boundary bigon is a bigon. %Moreover, for every boundary bigon $(\alpha,\beta)$, there exists
    %a point $x_i\in\cX$ such that $B(x_i,17r)$ contains the bigon $(D,\alpha,\beta)$.
  \end{lemma}
  \begin{proof}
    We know that $\ell(\gamma_{ij}),\ell(\gamma_{kl})<8r$, because $\dist(x_i,x_j)<4r$ and $\dist(x_k,x_l)<4r$. 
    From this it follows that $\gamma_{ij}\subset B(x_i,6r)$ and $\gamma_{kl}\subset B(x_k,6r)$.
    Indeed, if $x\in\gamma_{ij}$ is outside $B(x_i,6r)$, then, by triangle inequality
    $x\notin B(x_j,2r)$. Therefore, $\ell(\gamma_{ij})\ge \dist(x_i,x)+\dist(x,x_j)>8r$.

    As $\gamma_{kl}\cap\gamma_{ij}$ is not empty, we conclude that $\gamma_{kl}\cup\gamma_{ij}\subset B(x_i,12r)$.
    Since $12\sqrt{2}r<17r<R_0$, the map $\pi_{x_i}\colon B(x_i,17r)\to H_i$ 
    is well-defined and its image $U_i$
    contains the ball $B^{H_i}(x_i,12r)$; see \ref{item:pi}, \ref{item:U}.
    Let $\wh{\alpha},\wh{\beta}=\pi_{x_i}(\alpha),\pi_{x_i}(\beta)$. Then $\wh{\alpha}\cup\wh{\beta}$ is a simple closed curve on $H_i$ contained in $B^{H_i}(x_i,12r)$. By Jordan curve theorem, there exists a disk $\wh{D}$ in $B^{H_i}(x_i,12r)$ such that $\partial\wh{D}=\wh{\alpha}\cup\wh{\beta}$. The desired disk $D$ is obtained as $\phi_{x_i}(\wh{D})$. 
    By \ref{item:lip}, $D\subset B(x_i,12\sqrt{2}r)\subset B(x_i,17r)$.
  \end{proof}
  \begin{corollary}\label{cor:unique}
    Let $(\alpha,\beta)$ be a boundary bigon. If $D_1$, $D_2$ are two embeddings of a disk such that $(D_1,\alpha,\beta)$ and $(D_2,\alpha,\beta)$
    are bigons, then the images of $D_1$ and $D_2$ coincide.
  \end{corollary}
  \begin{proof}
    Suppose that $D_1$ and $D_2$ do not coincide. Their interiors are disjoint, for otherwise $\Sigma$ has self-intersections.
    Assume that $D_1\in B(x_i,17r)$ and $D_2\in B(x_{i'},17r)$, where $i'$ is any of the $i,j,k,\ell$ (we keep the notation
    of Lemma~\ref{lem:bigon_is_bigon}).
    Note that $\dist(x_i,x_j)<4r$ and $\dist(x_i,x_k),\dist(x_i,x_\ell)<12r$.
    Therefore, in the worst case scenario, when $D_2\subset B(x_k,17r)$ or $D_2\subset B(x_\ell,17r)$, we still have that
    $D_2\subset B(x_i,29r)$. Then $D_1\cup D_2$ glue to a sphere in $B(x_i,29r)$. But $29r<R_0$ and so $D_1\cup D_2$ belongs
     to a graph patch. This is impossible.
  \end{proof}
  Continuing the proof of Lemma~\ref{lem:G2}, we introduce more terminology.
  Let $D=(D,\alpha,\beta)$ be a bigon. We say that
  \begin{itemize}
    \item $D$ is \emph{minimal} if $D$ does not contain any smaller bigon;
    \item $D$ is \emph{desolate}  if $D$ does not contain any point $x_i$;
    \item $D$ is \emph{inhabitated} if $D$ contains at least one of $x_i$; see Figure~\ref{fig:des_inh}.
  \end{itemize}
  \begin{figure}
    \includegraphics[width=0.9\linewidth]{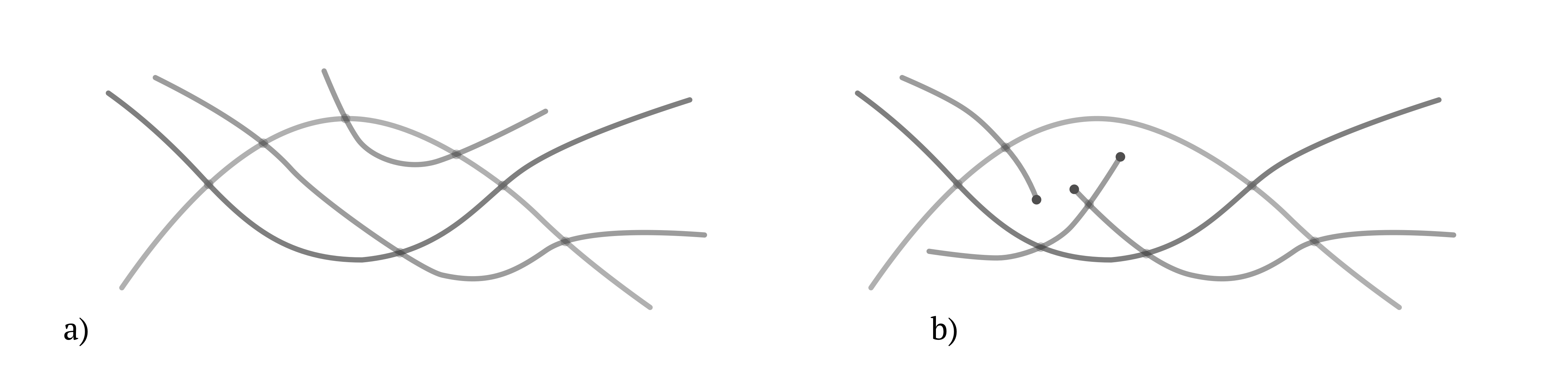}
    \caption{A desolate and an inhabitated bigon.}\label{fig:des_inh}
  \end{figure}

We will now describe a procedure called bigon removal;
see Figure~\ref{fig:bigon_removal}. Suppose $(D,\alpha,\beta)$
is a bigon. We can swap the roles of $\alpha$ and $\beta$, if needed,
to ensure that $\alpha$ is not longer than $\beta$.
  The curve $\beta$ is replaced by a curve $\beta'$ parallel to the curve $\alpha$. It is clear that the change
  can be made in such a way that the length of $\beta$ is not increased by more than $\varepsilon/2$.
    It might happen that one of the vertices $\alpha\cap\beta$ of the bigon is actually a starting
    point of the two curves in $\cG$, whose parts form a bigon. The picture in Figure~\ref{fig:bigon_removal}
    should be slightly altered, but the argumentation is the same.
  \begin{figure}[h!]
  \centering
  \includegraphics[width=0.9\linewidth]{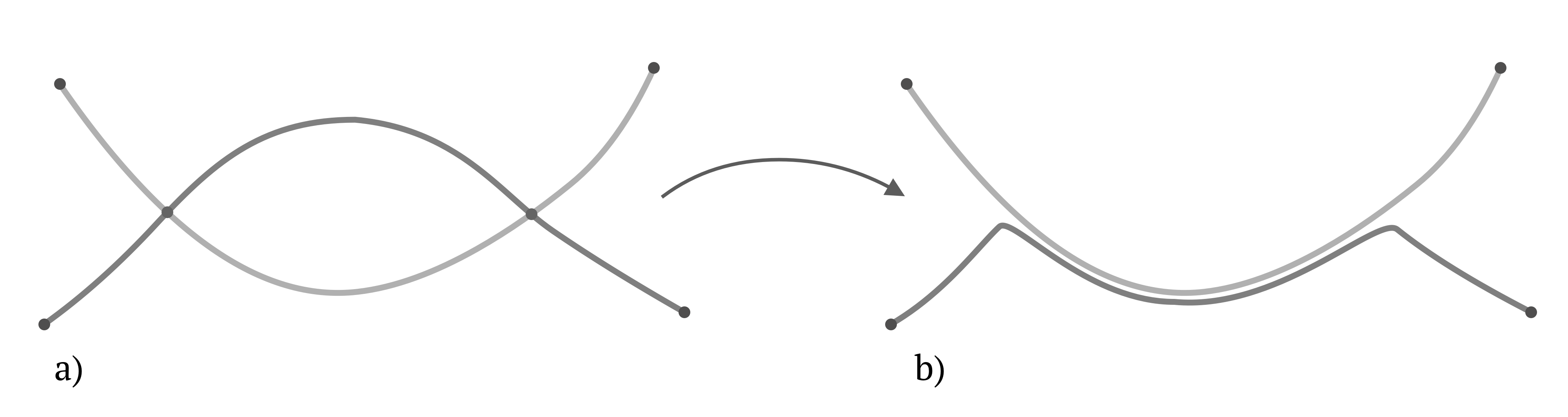}
  \caption{Bigon removal: (a) output bigon, (b) fragment of one curve replaced with the approximation of the second.}
  \label{fig:bigon_removal}
\end{figure}
\begin{lemma}\label{lem:bigon_removal}
  If $D$ is a desolate minimal bigon, bigon removal procedure applied to $D$ decreases the number of desolate bigons by $1$
  and creates no other bigons. 
\end{lemma}
\begin{proof}
  By Corollary~\ref{cor:unique}, 
  the number of bigons between two different curves $\gamma,\gamma'\in\cG$ is equal to $|\gamma\cap\gamma'|-1$. Therefore, we will strive to show that the number of intersection points between all curves in $\cG$
  decreases after bigon removal.

  Let $\gamma_{kl}$ and $\gamma_{st}$ be such that $\alpha\subset\gamma_{kl}$ and $\beta\subset\gamma_{st}$.
  The $\gamma'_{st}$ be the curve $\gamma_{st}$ with $\beta$ replaced by $\beta'$. We have
  \[|\gamma_{kl}\cap\gamma_{st}'|=|\gamma_{kl}\cap\gamma_{st}|-2.\]

  Suppose $\gamma_{uw}$ is another curve in $\cG$. If it does not hit the bigon $D$, we have that
  $\gamma_{st}\cap\gamma_{uw}=\gamma_{st}'\cap\gamma_{uw}$, so the number of intersection points is preserved.
  If $\gamma_{uw}$ hits the bigon $D$, we look at connected components of $\gamma_{uw}\cap D$. Each such connected
  component $\delta$
  is an arc, and if $|\delta\cap\alpha|=2$ or $|\delta\cap\beta|=2$, the arc $\delta$ and the relevant
  part of $\alpha$ or $\beta$ form a bigon contained in $D$, contradicting minimality of $D$. If $|\delta\cap(\alpha\cup\beta)|=1$,  one of the end points of $\delta$ is inside $D$, but such an end point must be
  an end point of $\gamma_{st}$, that is, it must be a point from $\cX$. This contradicts the condition
  that $D$ be desolate.

  The only remaining possibility is that  $|\delta\cap\alpha|=|\delta\cap\beta|=1$. This, in turn, shows that $|\gamma_{uw}\cap\alpha|=
  |\gamma_{uw}\cap\beta|$. Now, by construction $|\gamma_{uw}\cap\beta'|=|\gamma_{uw}\cap\alpha|$. Eventually
  $|\gamma_{uw}\cap\beta|=|\gamma_{uw}\cap\beta'|$, that is, $|\gamma_{uw}\cap\gamma_{st}'|=|\gamma_{uw}\cap\gamma_{st}|$. In other words, no new bigons are created.
\end{proof}
\begin{remark}
  The statement of Lemma~\ref{lem:bigon_removal} need not hold if $D$ is inhabitated. A bigon removal procedure
  can create new bigons, both inhabitated and desolate, whose number is rather hard to control; see Figure~\ref{fig:fail_to_remove}.
\end{remark}
\begin{figure}
  \includegraphics[width=0.9\linewidth]{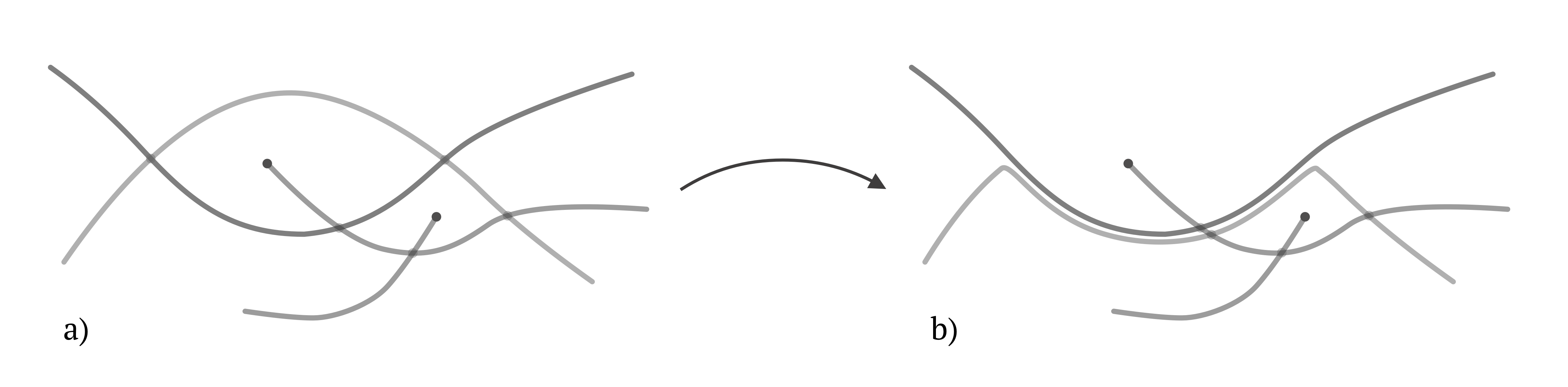}
  \caption{An attempt to remove an inhabitated bigon results in creating another inhabitated bigon. There is a little control on the number of bigons that can be produced in this way.}\label{fig:fail_to_remove}
\end{figure}
After a single bigon removal procedure, the length of one of the curves $\gamma_{ij}$ can increase,
we decrease $\varepsilon$ so that $\ell(\gamma_{ij})+\varepsilon<2\dist(x_i,x_j)$ for all $(i,j)\in\cI$. 

We now apply inductively the bigon removal procedure to all minimal desolate bigons, until there are no minimal desolate bigons. 
This requires a finite number of steps.
We make the following trivial observation.
\begin{lemma}
  If there are no minimal desolate bigons, there are no desolate bigons at all.
\end{lemma}
From now on we will assume that the set $\cG$ of curves is such that there are no desolate bigons.
The following lemma concludes the proof of Lemma~\ref{lem:G2}.
\begin{lemma}\label{lem:final_bound}
  Suppose curves $\gamma_{ij}$ and $\gamma_{kl}$ do not form desolate bigons. Then, the number of intersection
  points between $\gamma_{ij}$ and $\gamma_{kl}$ is bounded by $\cT(17)$.
\end{lemma}
\begin{proof}
  Suppose $\gamma_{ij}$ and $\gamma_{kl}$ are not disjoint. By Lemma~\ref{lem:bigon_is_bigon}, each bigon formed
  by $\gamma_{ij}$ and $\gamma_{kl}$ belongs to $B(x_i,17r)$. All such bigons have pairwise disjoint interiors.
  Moreover, each bigon is inhabitated. The number of bigons is bounded from above by the total number
  of points of $\cX$ in $B(x_i,17r)$, which is $\cT(17)$ according to Proposition~\ref{prop:sigma_bound}.

  The number of intersection points is not greater than the number of bigons. The lemma follows.
\end{proof}
The proof of Lemma~\ref{lem:G2} is complete.
\end{proof}

\section{Triangulation}\label{sec:triang_bound}
\subsection{Bounding number of triangles}\label{sub:bound_for_good}
\begin{proposition}\label{prop:triang_bound}
  Let $\cX$ be an $r$-net with $29r<R_0$.
  Suppose $\cG$ is a good tame collection of arcs. Then $\Sigma$ can be triangulated with
  at most $S(\Sigma)$ triangles with
  \begin{equation}\label{eq:S_sigma}S(\Sigma)=
    \frac{1}{12}\cT(4)^2\cT(8)^2\cT(12)^2\cT(15)^2\cT(17)^2|\cX|.
  %\frac{1}{48}\cT(4)^3\cT(8)^3\cT(12)^3\cT(17)^3\cT(22)^3|\cX|.
\end{equation}
\end{proposition}
\begin{proof}
  The proof of Proposition~\ref{prop:triang_bound} takes the rest of Subsection~\ref{sub:bound_for_good}.
  The triangulation is constructed by subdivision of $\Gamma=\bigcup\gamma_{ij}$. The vertices
  are going to be the points in $\cX$ as well as the intersection points $\gamma_{ij}\cap\gamma_{kl}$.
  We first bound the total number of intersection points of $\gamma_{ij}$. 
  \begin{lemma}\label{lem:cxbound}
    Suppose $x_i\in\cX$. The total number of points triples $j,k,l$
    such that $\gamma_{ij}\cap\gamma_{kl}\neq\emptyset$
    is less than or equal to $2\cT(4)\cT(8)\cT(12)$.
  \end{lemma}
  \begin{proof}
    First of all, number of indices $j$ such that $\dist(x_i,x_j)<4r$
    is at most $\cT(4)$. Next, a point $x\in\gamma_{ij}\cap\gamma_{kl}$
    has to lie at distance less than $4r$ from either $x_i$ or $x_j$. Suppose $\dist(x_i,x)<4r$.
    The same argument shows that either $\dist(x_k,x)<4r$ or $\dist(x_l,x)<4r$. Switch $k$ and $l$ so
    that $\dist(x_k,x)<4r$.
    Then $\dist(x_i,x_k)<8r$ and $\dist(x_i,x_l)<8r+\dist(x_k,x_l)<12r$.

    The total number of choices of $x_j$ is $\cT(4)$. The total number of choices of $x_k$ and $x_l$
    is $2\cT(8)\cT(12)$: the factor $2$ comes from chosing whether $\dist(x_i,x)<4r$ or $\dist(x_j,x)<4r$.
  \end{proof}
  Let now 
  \[\cZ=\bigcup_{(l,k)\neq (i,j)\neq (k,l)}\{\gamma_{ij}\cap\gamma_{kl}\}.\]
  Note that $\cX\subset\cZ$. Indeed, it is not hard to see that for any $i$ there are
  at least two points $j,l$ such that $(i,j),(i,l)\in\cI$ and then $x_i\in\gamma_{ij}\cap\gamma_{il}$.
  \begin{lemma}\label{lem:total_points}
    We have $|\cZ|<\frac12\cT(4)\cT(8)\cT(12)\cT(17)|\cX|$.
  \end{lemma}
  \begin{proof}
    Take $x_i\in\cX$. By Lemma~\ref{lem:cxbound} there are at most $2\cT(4)\cT(8)\cT(12)$
    configurations $j,k,l$ such that $\gamma_{ij}\cap\gamma_{kl}\neq\emptyset$.
    Therefore, the total number of quadruples $i,j,k,l$ such that $\gamma_{ij}\cap\gamma_{kl}\neq\emptyset$
    is $\frac12\cT(4)\cT(8)\cT(12)|\cX|$. The difference of factors $2$ and $1/2$ corresponds
    to the following observation: when summing over the indices $i$, each quadruple $i,j,k,l$ is actually counted four times. First, we can switch $i$ with $j$. Then
    we can switch the pairs $(i,j)$ and $(k,l)$.

    If two curves $\gamma_{ij}$ and $\gamma_{kl}$ intersect, by \ref{item:G2} the total number
    of intersections is at most $\cT(17)$. The lemma follows.
  \end{proof}
  The same argument yields
  \begin{lemma}
    Let $\sigma>0$ be such that $(\sigma+1/4)r<R_0$.
    Let $\cZ^{\sigma}_i$ be the number of points $z\in\cZ$ such that $z\in B(x_i,\sigma r)$. 
    Then $|\cZ^{\sigma}_i|<\frac12\cT(4)\cT(8)\cT(12)\cT(17)\cT(\sigma)$.
  \end{lemma}
  \begin{proof}
    We use the proof of Lemma~\ref{lem:total_points}. On passing from the number of triples $j,k,l$
    such that $\gamma_{ij}\cap\gamma_{kl}\neq\emptyset$ we multiply by the bound of the number of
    points in $B(x_i,\sigma r)$, that is, $\cT(\sigma)$, instead of the total number of points $x_i$,
    that is, $|\cX|$.
  \end{proof}
  % stopped here
  Now we pass to construction of triangulation. By the property \ref{item:G1}, each connected component
  of $\Sigma\setminus\Gamma$ is a subset of $\R^2$. 
 Take such a connected component $C$. Its boundary is a union of (parts of) curves
  $\gamma_{ij}$ intersecting at points of $\cZ$. We think of $C$ as a polygon with vertices $\cZ$, though we do not necessarily assume
  that $C$ has connected boundary.
  We triangulate this polygon by adding curves that connect vertices of $\cZ$. This provides us
  with a triangulation. In particular, the triangulation has
  the following property.
  \begin{property}\label{prop:erty}
An edge of the triangulation connects two elements in $\cZ$, which
belong to the closure of the same connected component of 
$\Sigma\setminus\Gamma$. %A triangle in the triangulation 
%belongs to a closure of a connected component of $\Sigma\setminus\Gamma$.
  \end{property}

  To estimate the number of triangles we use the following lemma.
  \begin{lemma}\label{lem:s_connected}
    Suppose $z_1,z_2\in\cZ$ and $z_1,z_2$ belong to the closure
    of the same connected component $S_z$ of $\Sigma\setminus\Gamma$.
    Then, there exists $x_i\in\cX$ such that $z_1,z_2\in\Sigma\cap B(x_i,15r)$.

   % Moreover, if $z_1,z_2,z_3$ form a triangle of the triangulation, 
   % then $z_1,z_2,z_3\in\Sigma\cap B(x_i,22r)$.
  \end{lemma}
  \begin{proof}
    Let $x_i$ be a point on $\Sigma$ such that $\dist(x_i,z_1)<r$.
    Set $y_1=\pi_{x_i}(z_1)$, $y_2=\pi_{x_i}(z_2)$.
    Consider the cycle $\Delta$ on $H_i$ as in Corollary~\ref{cor:delta_def}. By Corollary~\ref{cor:delta_def} there is a piecewise smooth
    simple closed curve $\Delta_0\subset\Delta$ such that $\Delta_0$
    cuts $H_i$ into two components: the component $S_i$, which is
    bounded and contains $B^{H_i}(x_i,r)$ and the other, which is unbounded. As $y_1\in B^{H_i}(x_i,r)$, the closure of $\pi_{x_i}(S_z)$
    must belong to $S_i$. In particular, $y_2$ cannot belong to $H_i\setminus\ol{S_i}$. Now, $\phi_{x_i}(\Delta)=\pi_{x_i}^{-1}(\Delta)$ is a subset
    of the union of the curves $\lambda_j$, all belonging to $B(x_i,15r)$. Therefore, $z_2\in B(x_i,15r)$ as well.

    %The second part follows from Property~\ref{prop:erty} and Corollary~\ref{cor:for_future_use}.
  \end{proof}
  \begin{corollary}\label{cor:triang}
    The total number of edges in the triangulation is bounded from above by 
    \[\frac18\cT(4)^2\cT(8)^2\cT(12)^2\cT(15)^2\cT(17)^2|\cX|.\]
  \end{corollary}
  \begin{proof}
    By Lemma~\ref{lem:s_connected} and Property~\ref{prop:erty}, any two edges connecting points $z_1$ and $z_2$ belong to the
    same $B(x_i,15r)\cap\Sigma$ for some $x_i\in\cZ$.
    The number of pairs $z_1,z_2$ of points $\cZ$ in $B(x_i,15r)$
    is equal to 
    \[\frac12|\cZ^{15}_i|^2\le\frac18\cT(4)^2\cT(8)^2\cT(12)^2\cT(15)^2\cT(17)^2.\]
    Summing up over all $x_i\in\cX$ we get the result.
  \end{proof}
  The rest of the proof of Proposition~\ref{prop:triang_bound} is straightforward.
    Each edge belongs to precisely two triangles and each triangle has
    three edges.
   % The second part uses Property~\ref{prop:erty} and Corollary~\ref{cor:for_future_use}: the number of triples of points in $\cZ^{22}_i$
    %is $\frac16|\cZ^{22}_i|^3$, the number of triangles in $B(x_i,22r)$ is less than the former. Each triangle belongs to some of the $B(x_i,22r)$.
\end{proof}
\subsection{Proof of Theorem~\ref{thm:main}}\label{sub:main}

Suppose $\cE^\ell_p(\Sigma)=E$ and let $C_0=2^{-1/2}\min(c_1,c_2^{-1/p\alpha})$; compare
Corollary~\ref{cor:regularity}. Here, $\alpha=1-2\ell/p$.
Set $r=\frac{1}{29}C_0 E^{-1/p\alpha}$.
Choose a net of points $\cX$ with this given $r$. By Proposition~\ref{prop:covering}, we have
\begin{equation}\label{eq:cx_bound}
|\cX|<C_1 A E^{2/p\alpha},
\end{equation}
where $C_1=2^{n/2+8}29^2V_2 C_0^{-2}$.

Let $\cG$ be a collection of good arcs connecting some of the pairs of points of $\cX$;
such collection exists by Proposition~\ref{prop:good_arcs}, because $4r<29r<R_0=C_0E^{-1/p\alpha}$.
The collection $\cG$ can be improved to a tame collection of good arcs by Lemmata~\ref{lem:G1}
and~\ref{lem:G2}, which work because $29r<R_0$.
A tame collection of arcs provides a triangulation with the number of triangles bounded above
by $S(\Sigma)$ triangles; see Proposition~\ref{prop:triang_bound}. In total, the number
of triangles is bounded above by $C_2 A E^{2/p\alpha}$, where
\[C_2=
\frac{1}{12}\cT(4)^2\cT(8)^2\cT(12)^2\cT(15)^2\cT(17)^2C_1.\]
%\frac{1}{48}\cT(4)^3\cT(8)^3\cT(12)^3\cT(17)^3\cT(22)^3 C_1.\]

By construction, each triangle is an image of a subset of a plane
$H_i$ (for some $i$) under the map $\phi_{x_i}$, which has
bounded derivative and bounded distortion by Corollary~\ref{cor:regularity} and Corollary~\ref{cor:distort}.
\def\MR#1{}
\bibliographystyle{abbrv}
\bibliography{research}
\end{document}